\title{On the genus of the complete tripartite graph $K_{n,n,1}$}

\date{2016-09-09}

  \documentclass[11pt]{article}

  \usepackage[utf8]{inputenc}

  \usepackage{amsmath}
  \usepackage{amssymb}
  \usepackage{latexsym}
  \usepackage{graphicx}

  \usepackage{color} 
  \usepackage{enumerate}
  \usepackage{setspace}
  \usepackage{cite} 

  \usepackage{hyperref}
  \usepackage{subcaption}

  \hypersetup{ 
    colorlinks=false,
    pdfborder={0 0 0},
  }
  \newenvironment{proof}{\noindent{\bf Proof\,}}{\hspace*{\fill}$\Box$}
  \newenvironment{proofof}[1]{%
  \noindent {\bf Proof of #1}}%
  {\hspace*{\fill}$\Box$}

  \newtheorem{theorem}{Theorem}[section]
  \newtheorem{lemma} [theorem] {Lemma}

\mdseries\itshape

\begin{document}

\author{Valentas Kurauskas\footnote{Vilnius University, Institute of Mathematics and Informatics, Akademijos 4, LT-08663 Vilnius, Lithuania.}
\footnote{\textcopyright 2016. This manuscript version is made available under the CC-BY-NC-ND 4.0 license \url{http://creativecommons.org/licenses/by-nc-nd/4.0/}. To appear in Discrete Mathematics, doi:10.1016/j.disc.2016.09.017.}
 }
\maketitle
\begin{abstract}
    For even $n$ we prove that the genus of the complete tripartite graph $K_{n,n,1}$ is $\lceil (n-1) (n-2)/4 \rceil$.
    This is the least number of bridges needed to build a complete $n$-way road interchange where changing lanes is not allowed.
    Both the theoretical result, and the surprising link to modelling road intersections are new. 
\end{abstract}

\section{Introduction}

We start with some basics of topological graph theory, for details and definitions that are not given here, see~\cite{grosstucker1987, moharthomassen2001}. 
Let $V(G)$ denote the set of vertices and $E(G)$ the set of edges of a graph $G$. 
The \emph{genus} of $G$, denoted $g(G)$, is the minimum number $g$ such that $G$ can be embedded on the orientable surface
$\mathbb{S}_g$, the sphere with  $g$ handles. If $G$ has genus $g$, there is a \emph{2-cell (cellular) embedding} of $G$ into $\mathbb{S}_g$, that is, an embedding where the interior of each region (face) is homeomorphic to an open disk.
A \emph{rotation system} of $G$ is a set $\{\pi_v: v\in V(G)\}$ where $\pi_v$ is a cyclic permutation of edges incident to $v$, called a \emph{rotation} at $v$. There is a well-known correspondence between orientable 2-cell embeddings and rotation systems: the rotation $\pi_v$ corresponds to the clockwise ordering of edges emanating from $v$ in the embedding.

Ringel \cite{ringel1965} showed that for any positive integers $n$ and $m$ the genus of the complete bipartite graph $K_{n,m}$ is $L(n,m)$, where $L(n,m)=\left \lceil \frac {(n-2)(m-2)} 4 \right \rceil$. An alternative proof was given by Bouchet~\cite{bouchet1978}. The inequality $g(K_{n,m}) \ge L(n,m)$ follows easily by Euler's formula $2 - 2g = v + f - e$, which holds for any 2-cell embedding of a graph with $v$ vertices, $e$ edges and $f$ faces into $\mathbb{S}_g$  
\cite{grosstucker1987}.


White \cite{white1965} conjectured in 1965 that the genus of a complete tripartite graph $K_{n,r,s}$ with $n\ge r \ge s$, satisfies
\begin{equation}\label{conj.white}
    g(K_{n,r,s}) = L(n,r+s) =  \left \lceil \frac {(n-2) (r + s - 2)} 4  \right \rceil.
\end{equation}
Since $K_{n,r+s}$ is a subgraph of $K_{n,r,s}$, we know that $g(K_{n,r,s})$ must be at least $L(n,r+s)$, see, for example, \cite{stahlwhite1976}. The challenge is to construct embeddings of such genus.

White's conjecture has been confirmed for complete tripartite graphs
with even part sizes \cite{harsfieldringel1990}, the graphs $K_{n, r, r}$, $n \ge r \ge 2$ and several other classes,  
see \cite{craft1998, ksz2004, stahlwhite1976}. However, it remains open in general. A corresponding conjecture for non-orientable embeddings
has been settled for all complete tripartite graphs by Ellingham, Stephens and Zha \cite{esz2006}. 



We prove
\begin{theorem}\label{thm.main}
    For any even positive integer $n$ there is a 2-cell embedding of $K_{n,n}$ 
    into $\mathbb{S}_{\lceil (n-1) (n-2)/4 \rceil}$ 
    which has a face bounded by a Hamiltonian cycle.
 \end{theorem}

This confirms (\ref{conj.white}) for even $n$ and $r$ with $n=r$ and $s=1$. Let $L(n) = \lceil (n-1) (n-2)/4 \rceil$.  For $n \ge r$ and $r$ even we can use the \textit{diamond sum} operation \cite{bouchet1978, ksz2004} to see that $g(K_{n,r,1}) \le L(n,r+1) + 1$ and (\ref{conj.white}) holds if either $r \bmod\,4 = 2$ or both $r \bmod\, 4 = 0$ and $n \bmod\,4 \in \{0,1\}$. The same technique easily yields $g(K_{n,n,1}) \le L(n) + 1$ for any odd~$n$.

To our knowledge,
prior to this work, $g(K_{n,n,1}) = L(n)$ and, implicitly, Theorem~\ref{thm.main}, has been proved 
only 
for an infinite sequence of $n$ of the form $3^q (2^p + \frac 1 2) + \frac 1 2$ where $q\ge0$ and $p\ge 3$ are integers.
This follows from the embeddings of $K_{n+1}$ where all faces are bounded by Hamiltonian cycles, constructed by Ellingham and his co-authors \cite{ellinghamstephens2009, ellinghamschroeder2014}.

The authors in \cite{esz2006} claimed a proof of (\ref{conj.white}) for a very general family of graphs $K_{n,r,s}$, including the cases studied in this paper. However, no proofs have appeared since the publication of \cite{esz2006} in 2006.
Our idea is simple, original and does not involve the traditional voltage graph \cite{ellinghamschroeder2014, grosstucker1987} or transition graph constructions \cite{ellinghamstephens2009, esz2006}.

We conjecture that Theorem~\ref{thm.main} also holds for any odd integer $n \ge 3$. The main result and the application described below leads to a more general question: when, among all minimum genus embeddings of a graph $G$ there is one with a Hamiltonian cycle bounding a face?

\section{Road interchanges}


Our work is motivated by a beautiful road junction optimisation problem, which, to our knowledge, is being described
here for the first time. 

\begin{figure}[t] 
    \begin{center}
    \begin{subfigure}{.45\textwidth}
           \centering
           \includegraphics[height=0.8\textwidth]{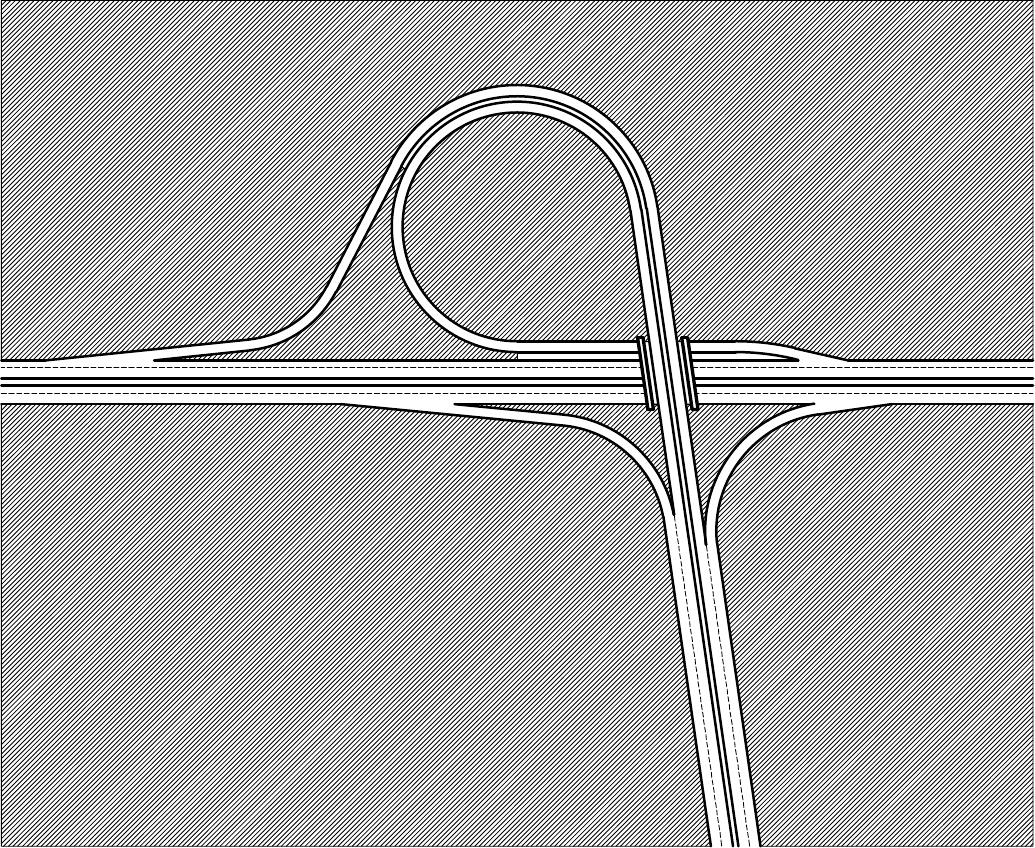}
           \caption{}
           \label{fig.trumpet}
        \end{subfigure}
        \begin{subfigure}{.45\textwidth}
           \centering
           \includegraphics[height=0.8\textwidth]{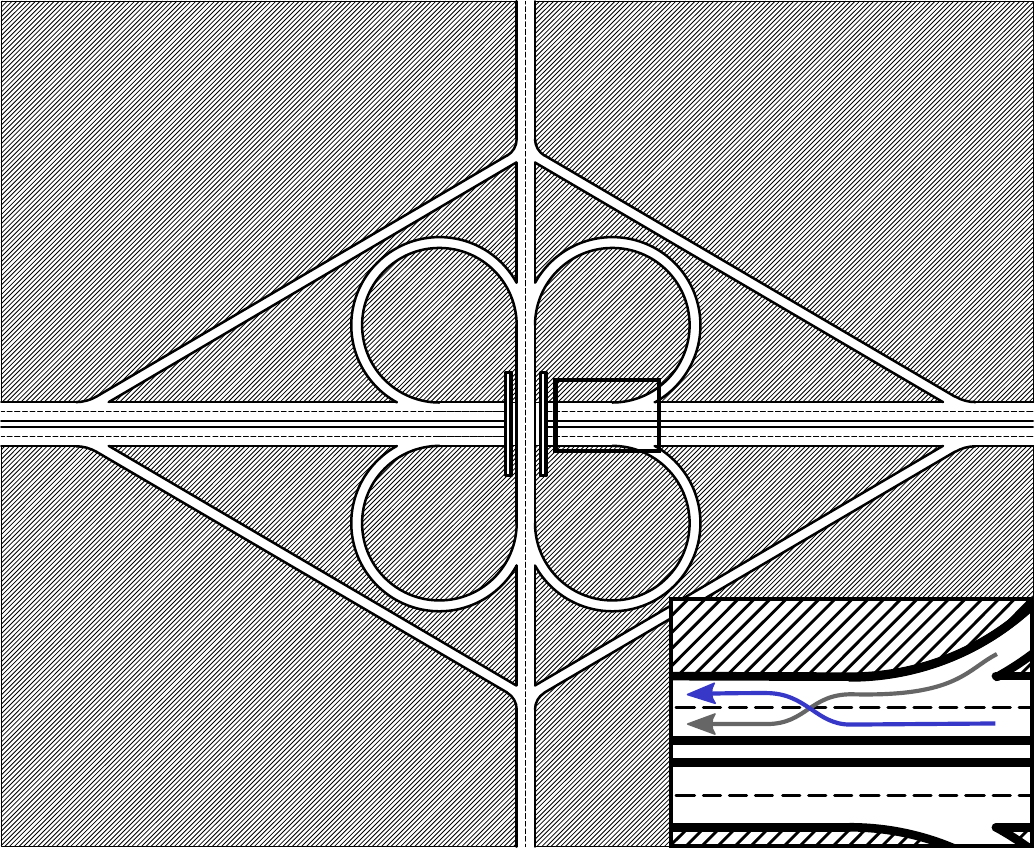}
           \caption{}
           \label{fig.cloverleaf}
        \end{subfigure}
        \begin{subfigure}{.45\textwidth}
           \centering
           \includegraphics[height=0.8\textwidth]{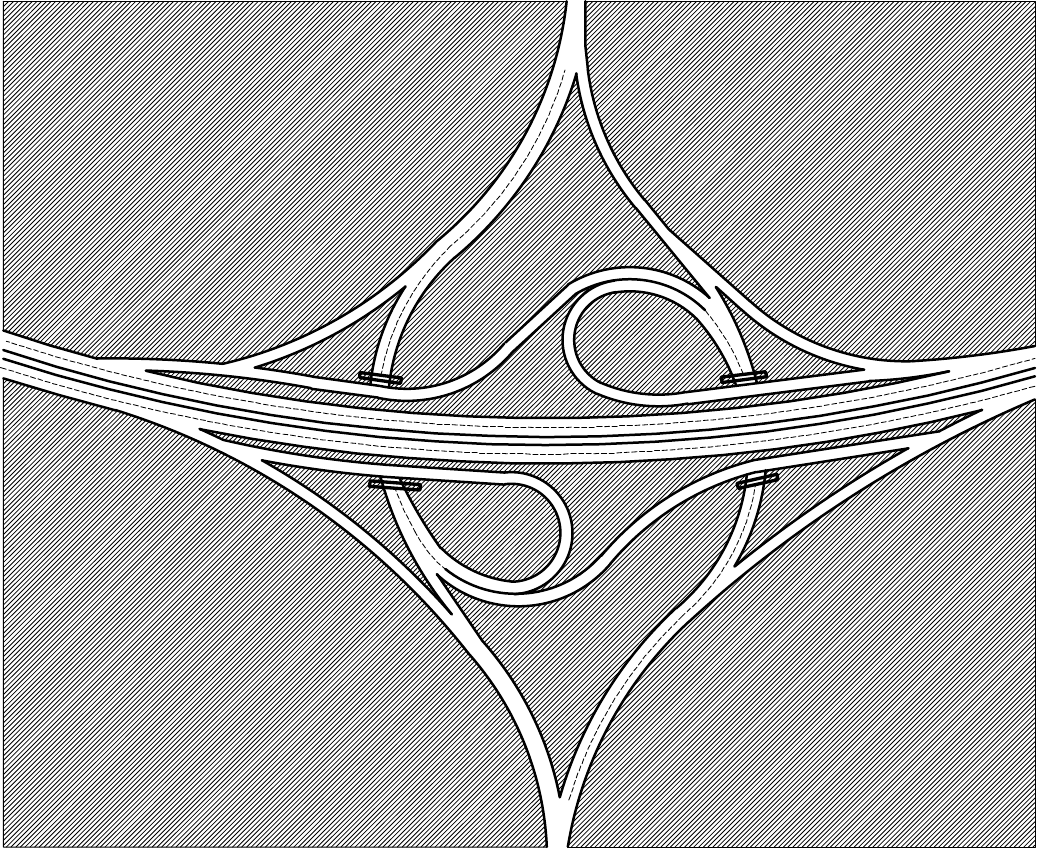}
           \caption{}
           \label{fig.d_trumpet}
        \end{subfigure}
        \begin{subfigure}{.45\textwidth}
           \centering
           \includegraphics[height=0.8\textwidth]{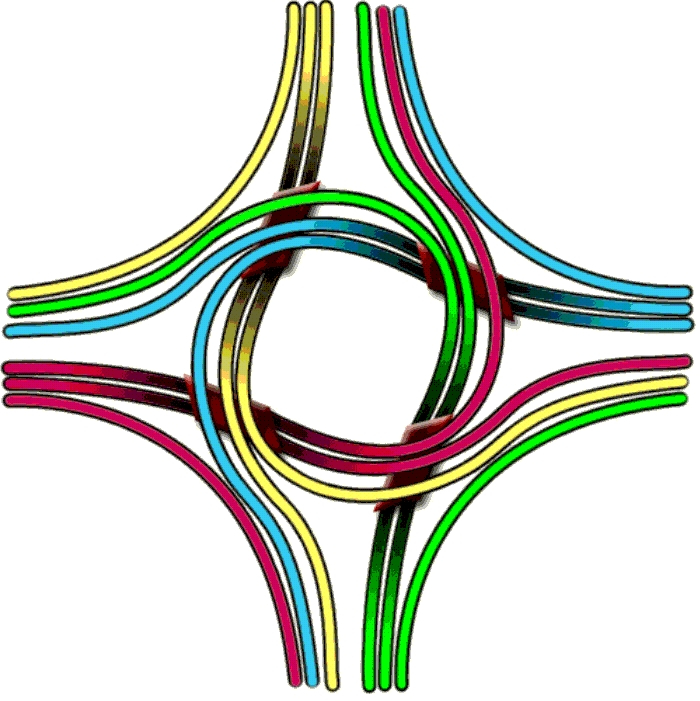}
           \caption{}
           \label{fig.pinavia}
        \end{subfigure}

    \end{center}
    \caption{Some junction types: (a) trumpet, (b) cloverleaf (an example of traffic weaving in the magnified rectangle), (c) double trumpet, (d) Pinavia \cite{jbk2010}.
    Images (a)-(c) from Davies and Jokiniemi \cite{daviesjokiniemi2008} used with authors' permission; image (d) from \textit{\url{www.pinavia.com}} used with authors' permission.}
        \label{fig.junctions}
    \end{figure}


Many types of road interchanges are known by engineers and built in practice, see, for example, Chapter~7 of \cite{garberhoel2009}. 
A popular design is the 4-way \textit{cloverleaf interchange}.
A 3-way example is the \textit{trumpet interchange}, see Figure~\ref{fig.junctions}.
To drive through certain junctions, some vehicles must cross each other's path and
change their lanes.
For example, drivers approaching a cloverleaf from the south and
going west, need a maneuver similar to the one depicted
in a corner of Figure~\ref{fig.cloverleaf} (assuming right-hand traffic). 
In engineering, a situation when 
a ``vehicle first merges into a stream of traffic,
obliquely crosses that stream, and then merges into a second stream moving in the same direction''
is called \textit{traffic weaving} \cite{garberhoel2009}.

Traffic weaving, which is generally undesirable, can be avoided in the trumpet, the \textit{all-directional four leg}~\cite{garberhoel2009}, also called \textit{four-level stack}, and, for example, recently invented \emph{Pinavia} interchanges (Figure~\ref{fig.pinavia}).
In these interchanges 
the lanes can be completely separated so that
the exit motorway and lane of a vehicle are determined by the lane it enters the junction (lane changes inside the junction are not necessary or not allowed). We call such interchanges \textit{weaving-free}. 


\hyphenation{Sko-pen-kov}
\hyphenation{Kra-sau-skas}

Let $n\ge 2$ be an integer. Based on an idea by Rimvydas Krasauskas (personal communication) and Mikhail Skopenkov, 
see also \cite{kotov1983}, we
propose to model a \emph{weaving-free $n$-way interchange} by a quadruple $(G, H, \mathcal{M}, \mathbb{S})$.
Here $G$ is a bipartite multigraph with $n$ white and $n$ black vertices as its parts, 
$H$ is a directed Hamiltonian cycle on which the vertex colours alternate,
$\mathbb{S}$ is a closed connected orientable surface and $\mathcal{M}$ is an embedding of $G$ into $\mathbb{S}$
such that $H$ bounds a \emph{face} (a region homeomorphic to an open disc). 
The $i$-th motorway is represented by one white vertex $a_i$ (the incoming direction) and one black vertex $b_i$ (the outgoing direction).
The cycle $H$ corresponds to the order the motorways enter and leave the junction. In particular, if traffic is right-hand and the clockwise order in which the motorways join the junction is $(1, \dots, n)$, then $H=(a_1, b_1, \dots, a_n,b_n)$.
Finally, the connections between the ingoing and outgoing lanes are represented by the remaining edges $uv \in E(G)$ such that $uv \not \in E(H)$.
For example, the trumpet interchange corresponds to the embedding shown in Figure~\ref{fig.trumpetgraphemba}.

\begin{figure}
    \begin{center}
    \begin{subfigure}{.3\textwidth}
           \centering
           \includegraphics[scale=0.2]{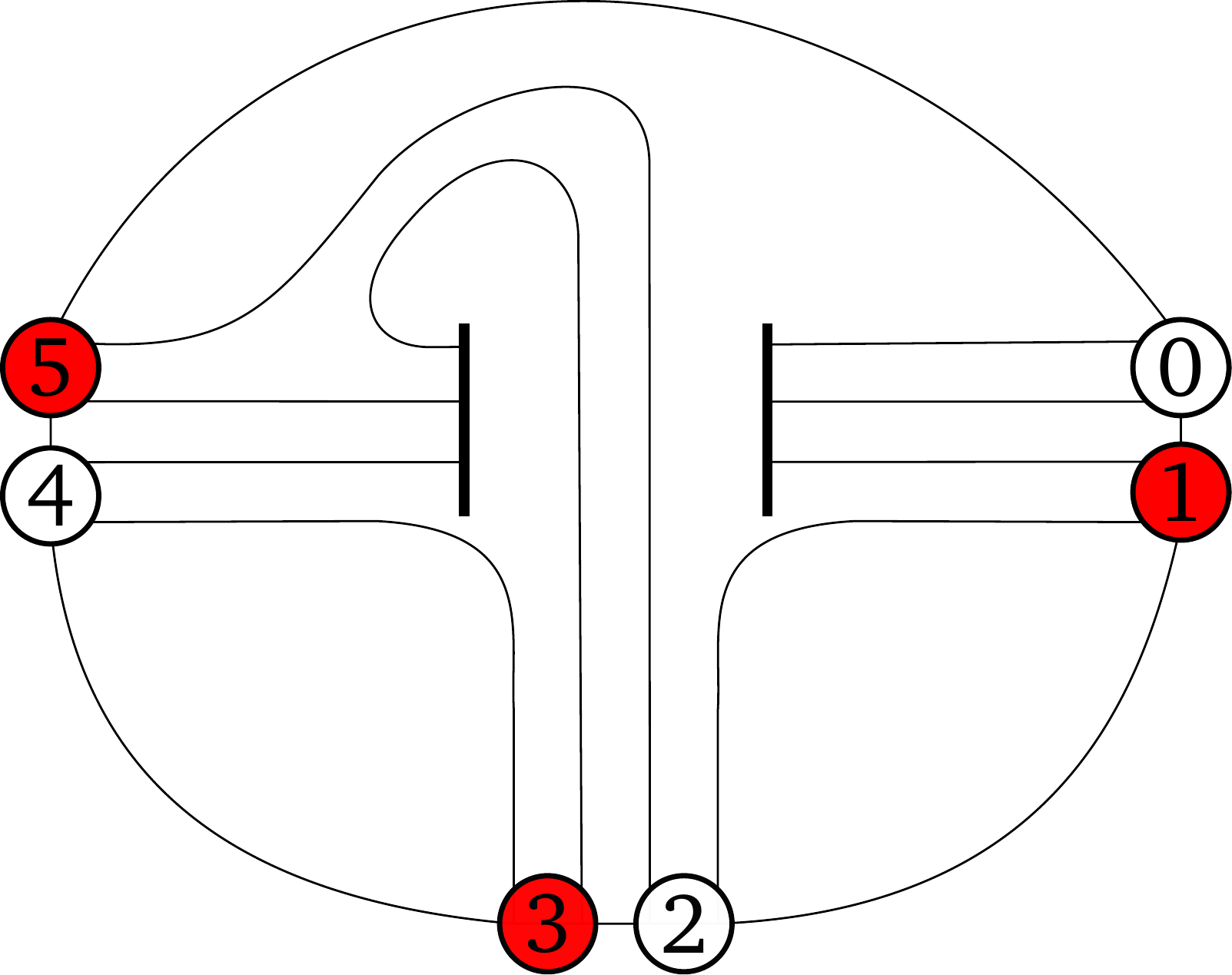}
           \caption{}
           \label{fig.trumpetgraphemba}
    \end{subfigure}
    \begin{subfigure}{.3\textwidth}
           \centering
           \includegraphics[scale=0.2]{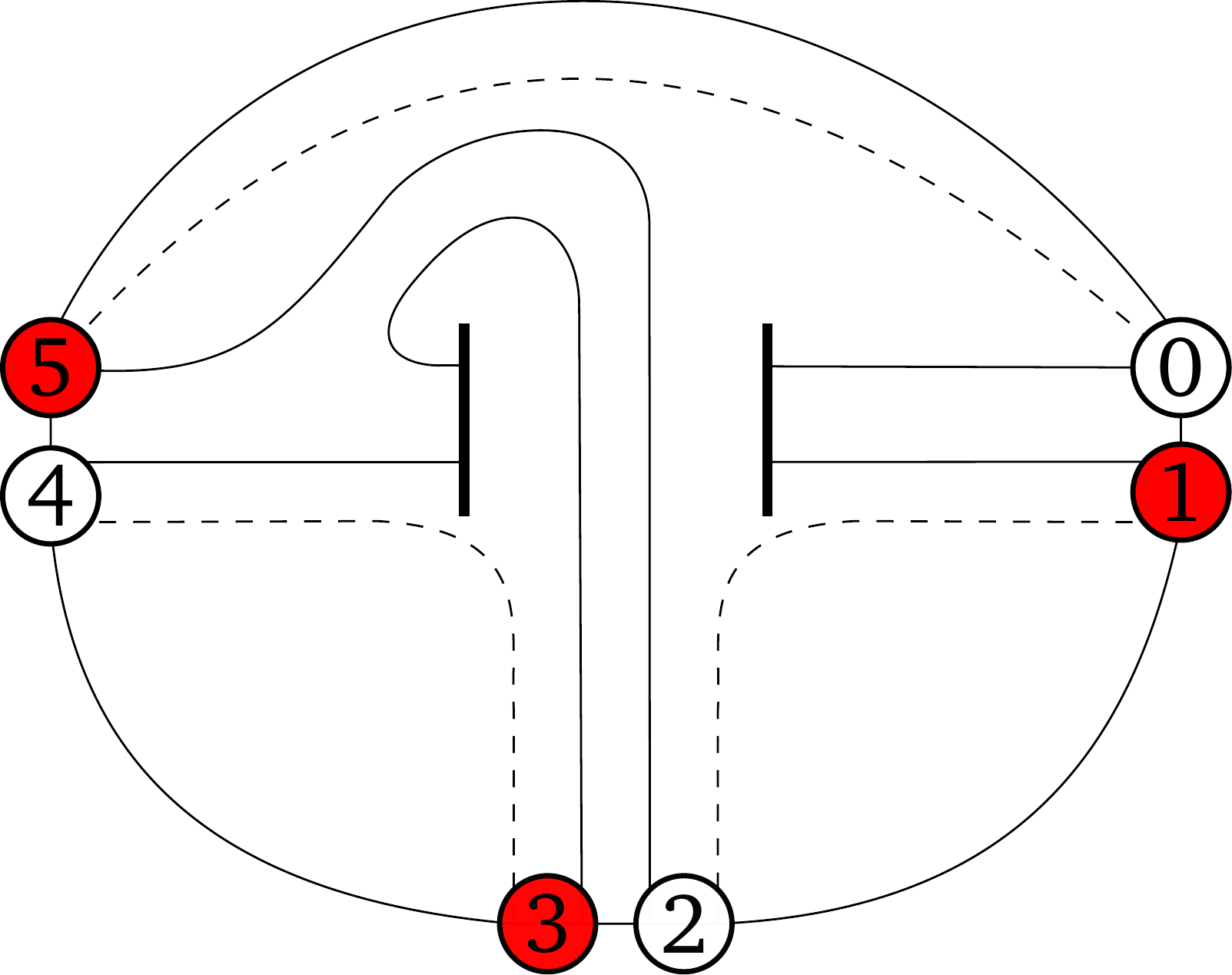}
           \caption{}
           \label{fig.trumpetgraphembb}
    \end{subfigure}

    \end{center}
        \caption{(a) Representation of the trumpet interchange by a multigraph embedded into a genus 1 surface; $H=(0,1,2,3,4,5)$. (b)
            A different complete interchange, an embedding of $K_{3,3}$, shown in solid lines.
            Lanes 
            connecting every pair of motorways 
            are obtained by duplicating appropriate edges of $H$ (dashed lines).
         }
        \label{fig.trumpetgraphemb}
\end{figure}


The \textit{number of bridges} in an interchange $I = (G, H, \mathcal{M}, \mathbb{S})$ is defined to be the genus of $\mathbb{S}$.
We call $I$ \textit{complete} if for each white vertex $u$ and each
black vertex $v$ there is an edge $uv \in E(G)$. 
If $I$ is complete we can iteratively remove repeated edges from $G$ which do not lie on $H$ to obtain
a complete interchange $I'=(G',H,\mathcal{M}, \mathbb{S})$ on the same surface with $G'$ a complete bipartite graph $K_{n,n}$, see Figure~\ref{fig.trumpetgraphembb}.
Similarly we can insert new lanes connecting arbitrary pairs of motorways, for example, by duplicating edges of $G$ without changing $\mathbb S$. Therefore we focus on complete interchanges with $G$ isomorphic to $K_{n,n}$ below.


Krasauskas was interested in the minimum number of bridges of a complete $n$-way weaving-free interchange, for a given $n\ge2$. 
We call the interchanges that achieve this minimum \textit{optimal} for $n$.
By the next lemma, Theorem~\ref{thm.main} provides a solution for all even $n$. 


\begin{lemma} \label{lem.interemb}
    Let $G$ be a complete bipartite graph $K_{n,n}$.
    An interchange $(G, H, \mathcal{M}, \mathbb{S})$ is optimal for $n$ if and only if
    $\mathcal{M}$ minimizes the genus over 2-cell embeddings of $G$ which have
    a face bounded by a Hamiltonian cycle.
   %
\end{lemma}

\begin{proof}
    Suppose $I$ is an optimal interchange for $n$. Let $I = (G,H,\mathcal{M}, \mathbb{S})$.
    It suffices to prove that $\mathcal{M}$ is 2-cell. Suppose the contrary.
    
    Youngs, see paragraph 3.5 of \cite{youngs1963}, showed, 
    that in this case the surface $\mathbb{S}$ can be transformed into a surface $\mathbb{S}'$,
    such that $\mathcal{M}$ is a 2-cell embedding, when viewed as an embedding from $G$ to $\mathbb{S}'$,
    all 2-cell faces of $\mathcal{M}$ are preserved (in particular, the face bounded by $H$)
    and the genus of $\mathbb{S}'$ is smaller than the genus of $\mathbb{S}$. This is a contradiction
    to the optimality of $I$.
\end{proof}

\section{Solutions for small $n$}

Given a graph $G$ with a rotation system $\{\pi_v: v \in V(G)\}$, the boundary circuit of each
face in the corresponding 2-cell embedding can be obtained by a simple \emph{face-tracing algorithm} \cite{grosstucker1987}. This algorithm repeats the following procedure until all edges are traversed (once in both directions).
Start with an arbitrary unvisited 
edge $e_0 = (v_0,v_1)$. Then for $i \ge 1$ if $e_i=(v_{i-1}, v_{i})$ define $e_{i+1} = \pi_{v_i}(e_i) = (v_i, v_{i+1})$, 
that is, at each vertex make a clockwise turn. Repeat this until $e_{t+1}=e_0$ for some $t > 0$, which closes a directed walk $(v_0, v_1, \dots, v_t)$. 

Once we know the number of faces, Euler's formula, yields the genus of the surface. 
Thus, we can find optimal interchanges for small $n$ by 
running the above algorithm 
for all rotation systems of $K_{n,n}$ with a face bounded by a fixed Hamiltonian cycle. The number of such rotation systems, $(n-2)!^{2n}$, grows with $n$ fast, the embeddings with minimal genus are very rare and exhaustive search is feasible only for $n\le 5$. 
Using two different methods we verified Theorem~\ref{thm.main} for all $n\in\{2,\dots, 11\}$. 
The first method was a randomized version of \cite{bouchet1978}, see also \cite{ksz2004, moharthomassen2001}.
The second one was restricting the exhaustive search to certain symmetry patterns; this yielded more symmetric solutions, especially for even $n$.
One of the two solutions for $n=4$, see Figure~\ref{fig.opt_4_0}, 
is known to engineers as the \textit{double trumpet} interchange. It has been built in practice and has $L(4) = 2$ bridges (Figure~\ref{fig.d_trumpet}).
The other solution, see Figure~\ref{fig.opt_4_1}, has the same rotation system as the four-level stack interchange. However, the construction 
used in practice
see, for example, \cite{garberhoel2009},
corresponds to a non-cellular embedding $\mathcal{M}$ and a surface $\mathbb{S}$ of genus larger than two.

\begin{figure}
    \begin{center}
    \begin{subfigure}{.3\textwidth}
           \centering
           \includegraphics[scale=0.2]{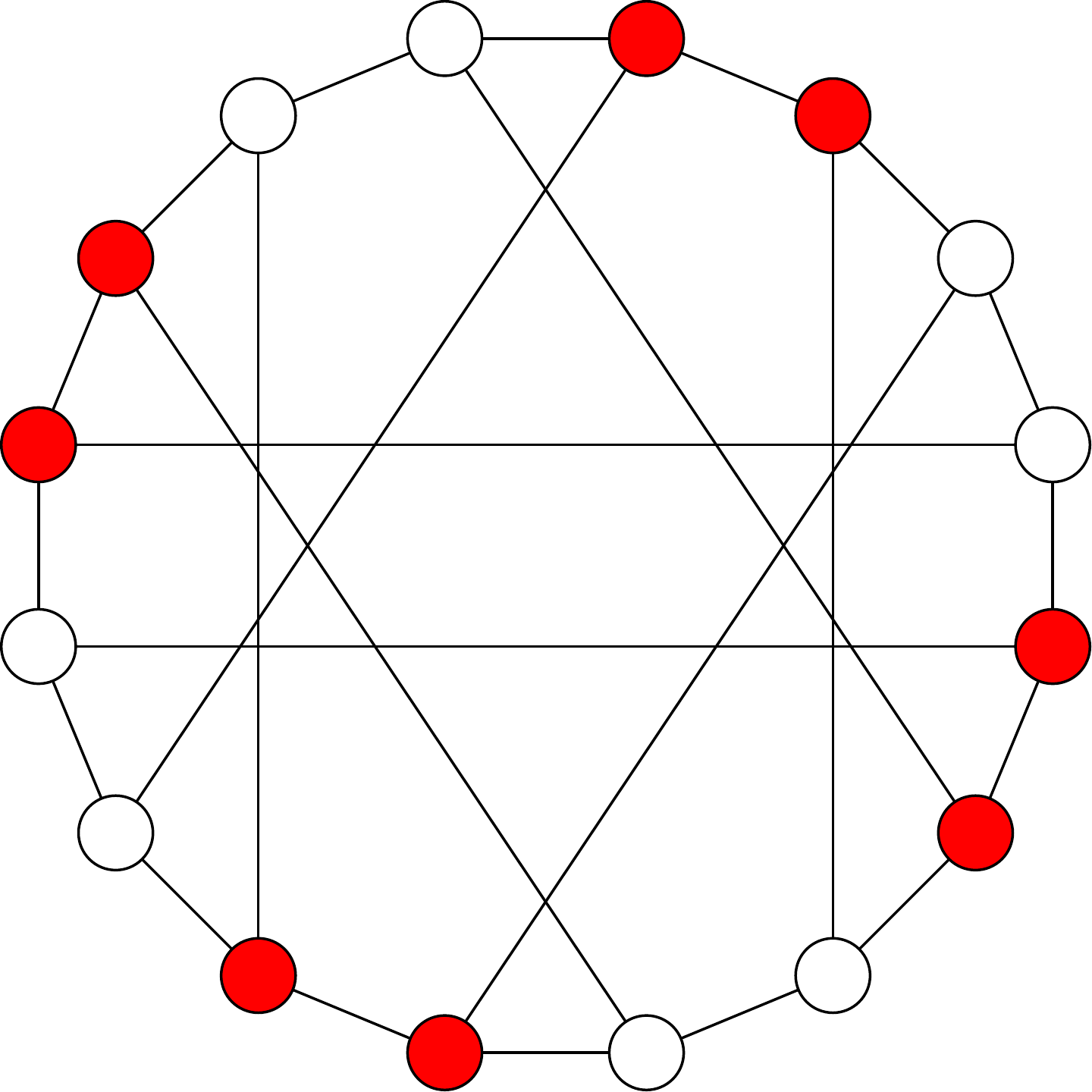}
           \caption{}
           \label{fig.opt_4_0}
    \end{subfigure}
    \begin{subfigure}{.3\textwidth}
           \centering
           \includegraphics[scale=0.2]{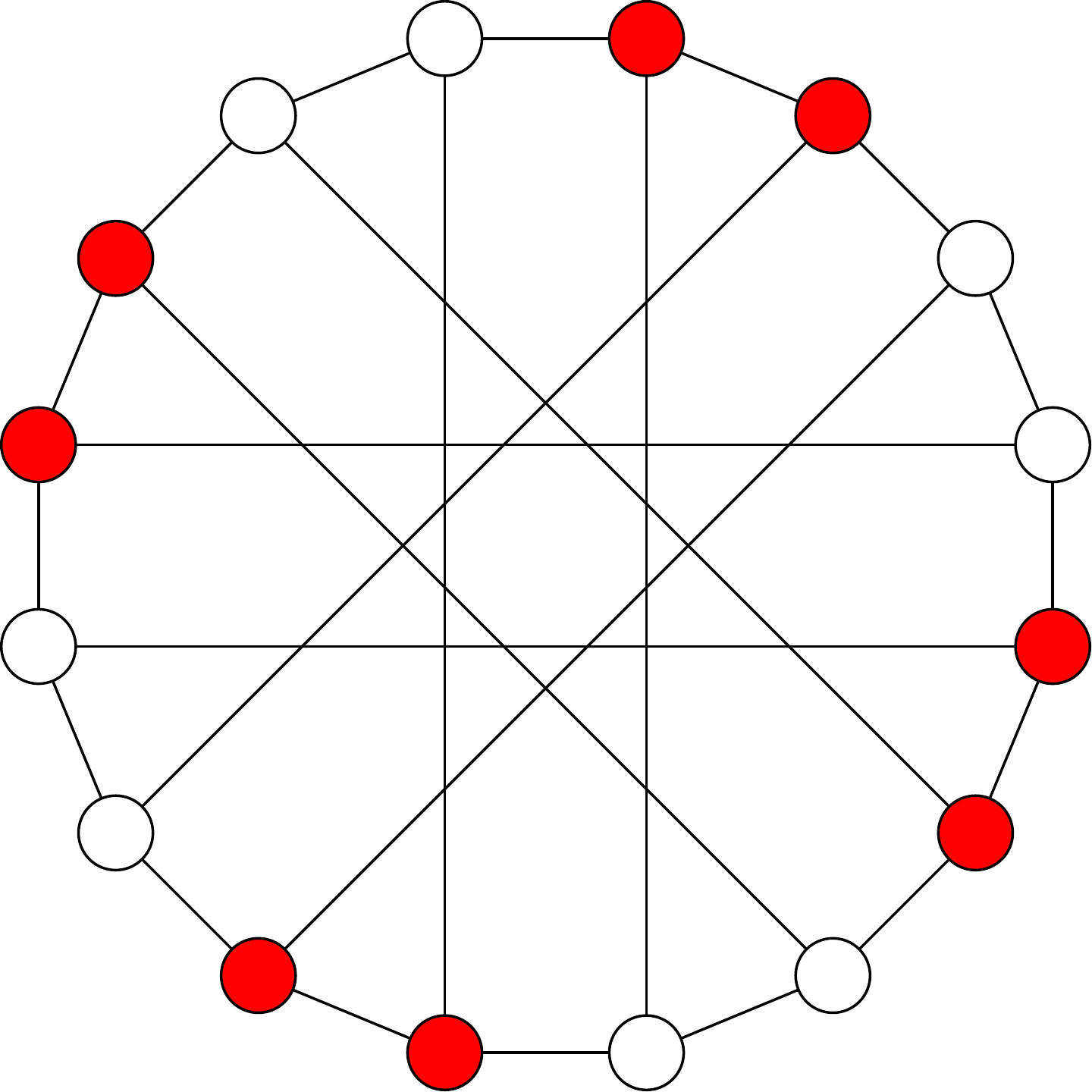}
           \caption{}
           \label{fig.opt_4_1}
    \end{subfigure}
    \begin{subfigure}{.3\textwidth}
           \centering
           \includegraphics[scale=0.2]{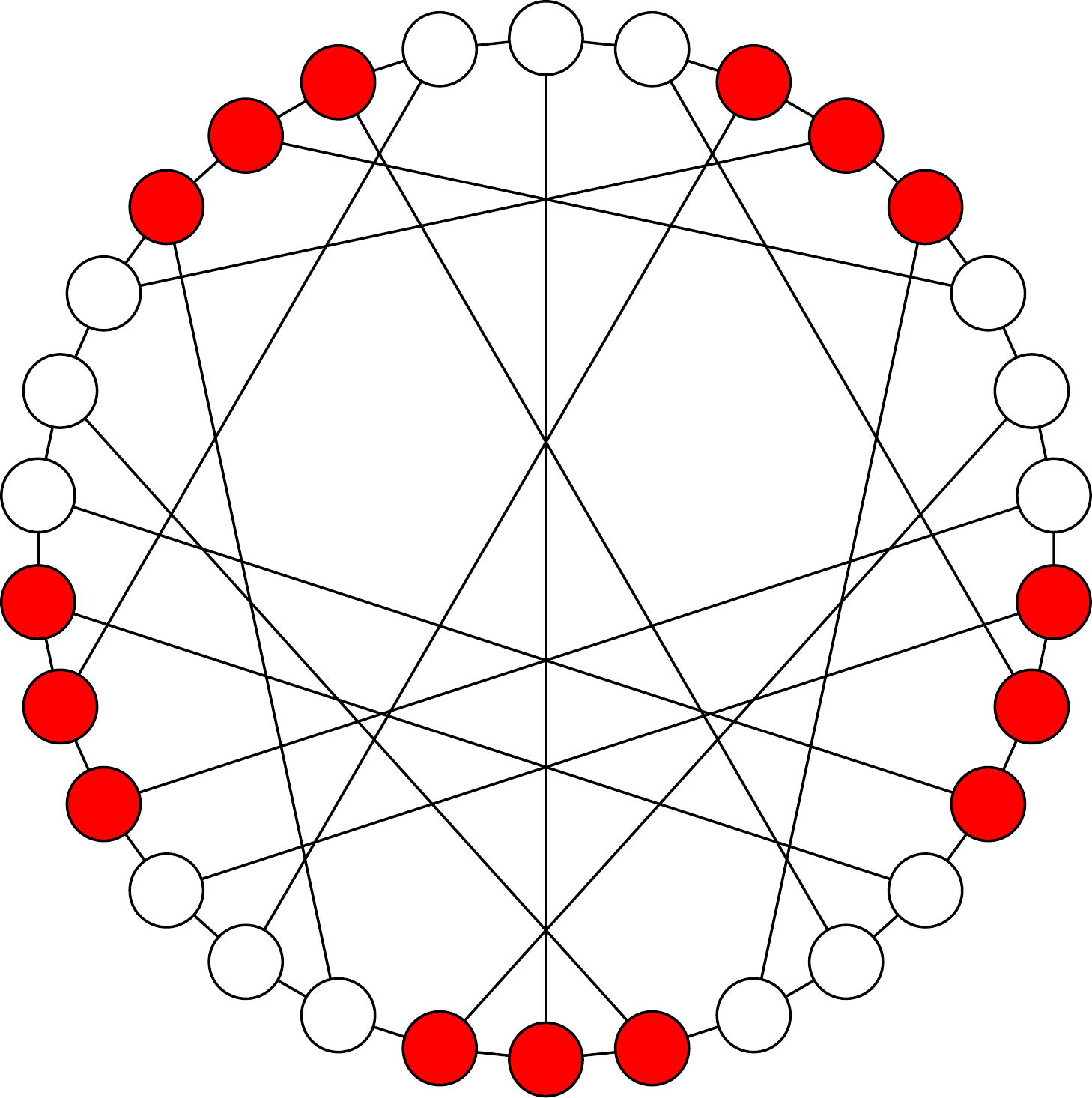}
           \caption{}
           \label{fig.opt_5_0}
    \end{subfigure}
    \end{center}
    \caption{The unique (up to isomorphism) optimal interchanges for $n=4$ and $n=5$ expanded to Hamiltonian cubic graphs.
             To get the rotation system for $K_{n,n}$, contract same colour paths on the outer circle.}
        \label{fig.opt_small}
\end{figure}

\section{Proofs}

In the proof we work only with 2-cell embeddings of simple graphs. Therefore we use terms \emph{embedding} and \emph{rotation system} synonymously and we use \emph{face} to refer to \emph{boundary circuit}. The genus of an embedding is the genus of the corresponding surface.
Also, we represent a rotation  $\pi_v$  as a cyclic permutation of vertices (rather than edges) incident to $v$  and use the notation $v: \pi_v$. For example, the rotation $v: xyz$ or $\pi_v = (x,y,z)$ represents 3 directed edges (arcs) emanating from $v$ in this clockwise order: $((v,x), (v,y), (v,z))$. 


We will need the next simple result. 
\begin{lemma}\label{lem.Knm} (Ringel \cite{ringel1965})
    Let $n$ and $m$ be even. The following rotation system gives a minimum genus embedding of $K_{n,m}$
    with parts $\{u_s: s \in \mathbb{Z}_m\}$ and $\{v_t: t \in \mathbb{Z}_n\}$:
    \begin{align*}
        &u_{2j}: v_{n-1} v_{n-2} \dots v_0;  
        \qquad u_{2j+1}: v_0 v_1 \dots v_{n-1}; \qquad j = 0, \dots, m/2 - 1;
        \\    &v_{2k}: u_0 u_1 \dots u_{m-1}; \qquad
        v_{2k+1}: u_{m-1} u_{m-2} \dots u_0; \qquad k = 0, \dots, n/2-1.
 \end{align*}
\end{lemma}

\begin{proof}
    The faces of the proposed embedding are
    \begin{align*}
        \{v_t u_{s+1} v_{t+1} u_s: s \mbox{ even}, t \mbox{ even}\} \cup  \{u_s v_{t+1} u_{s+1} v_t: s \mbox{ odd}, t \mbox{ odd}\}.
    \end{align*}
    This is a minimum genus embedding of $K_{n,m}$ because all faces are of length 4, which is minimum possible for a bipartite graph. Indeed, there are $v = 2n$ vertices, $e = n m$ edges and $f = \frac {nm}2$ faces in total. By Euler's formula the genus $g$ satisfies $2 - 2g = v + f - e$, or $g = (n-2) (m-2) / 4 = L(n,m)$.
\end{proof}

\medskip

The corresponding surface can be visualised as a book with $\frac n 2$ thick leaves, where each leaf has $\frac m 2 - 1$ arcs, see Figure~\ref{fig.arcs}. For $k \in \{0, \dots, \frac n 2 - 1\}$, glue faces $2k$ and $2k+1$ (bounded by the black edges in the figure) along the alternating edges $u_{m-1} u_{m-2}, u_{m-3} u_{m-4},$ $\dots,$ $u_1 u_0$. Also glue faces $2k$ and $(2k - 1) \bmod (2n)$ along the remaining edges $u_{m-2} u_{m-3},$ $u_{m-4} u_{m-5},$ $\dots,$ $u_2 u_1$, forming a ``leaf''.   Each of the $n$ faces is bounded by a Hamiltonian cycle. Now, as in Bouchet \cite{bouchet1978}, the embedding of Lemma~\ref{lem.Knm} results by placing a vertex $v_j$ inside each face $j$ and connecting it to each vertex $u_i$ (grey edges in the picture).

\begin{figure}
    \centering
    \includegraphics[scale=0.5]{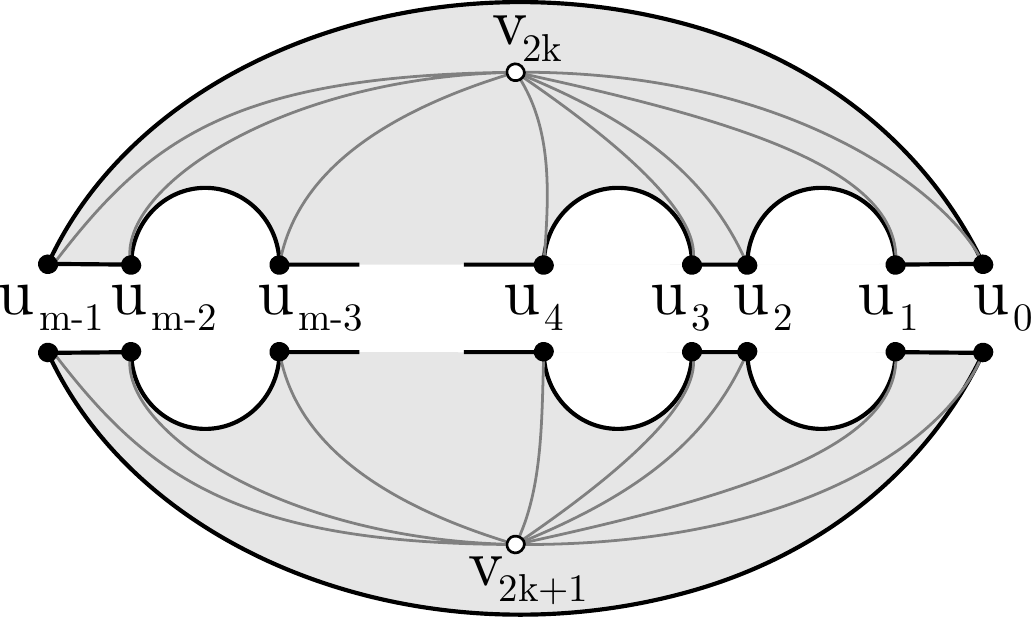}    
    \caption{The surface of Lemma~\ref{lem.Knm} is a ``book'' with arched leaves.}
    \label{fig.arcs}
\end{figure}

Note that for $m=n$ the embedding of Lemma~\ref{lem.Knm} contains a family
\begin{equation} \label{eq.cylinder}
    \mathcal{F}(n) = \left\{v_s u_{s+1}v_{s+1} u_s: s \mbox { even} \right\}
\end{equation}
of $n/2$ disjoint faces that covers all of the vertices. 
\begin{equation} \label{eq.cylinderprime}
    \mathcal{F}'(n) = \left\{u_s v_{s+1}u_{s+1} v_s: s \mbox { odd} \right\}
\end{equation}
is another such family. The union of $\mathcal{F}(n)$ and $\mathcal{F}'(n)$ makes up a cylindrical band of rectangles glued along their opposite sides, see Figure~\ref{fig.circle8a}. Below, given an embedding of $K_{n,n}$ as in Lemma~\ref{lem.Knm}, we call faces in $\mathcal{F}(n)$ \emph{special}.


We prove Theorem~\ref{thm.main} by combining 4 minimum genus embeddings of complete bipartite graphs with equal part sizes.
Our key insight comes after a careful analysis of symmetric
optimal 
$n$-way interchanges 
for $n=4$ and $n=8$ generated by computer. 

\bigskip

\begin{proofof}{Theorem~\ref{thm.main}.} 
    We aim to 
    construct 
    an
    embedding of a complete bipartite graph
    with parts
    $\{0, 2, \dots, 2n-2\}$ and $\{1,3, \dots, 2n-1\}$,
    such that one of the faces is the Hamiltonian cycle
    $(0,1,\dots,2n-1)$.
    We view the vertex set of the resulting graph as $\mathbb{Z}_{2n}$ and perform
    arithmetics in the proof modulo $2n$. 

   \emph{The case $n \equiv 0 \;(\bmod\; 4)$.}
   We start by partitioning the edges of $H$ into four parts $P_i$, $i \in \{0,1,2,3\}$, 
   where 
   \[
       P_i = \{ (t-1, t): t \equiv i \, (\bmod \,4),  t \in \mathbb{Z}_{2n} \}
   \]
   For each $i \in \{0, 1, 2, 3\}$ the pairs in $P_i$ cover $n/2$ even and $n/2$ odd vertices. 
   We denote these sets by $U_i = \{u: uv \in P_i\}$ and $V_i = \{v: uv \in P_i\}$.
   
   For each $i$ we apply Lemma~\ref{lem.Knm} to get an embedding $\mathcal{R}_i$ of $G_i$, where $G_i$ is $K_{n/2,n/2}$ with parts $U_i$ and $V_i$. To apply the lemma, it suffices to choose a permutation $\bar{u}_i=u_{i,0}u_{i,1}\dots u_{i, \frac n 2-1}$ of $U_i$ as  $u_0 u_1 \dots u_{\frac n 2 -1}$ and 
   a permutation $\bar{v}_i = v_{i,0}v_{i,1}\dots v_{i, \frac n 2-1}$ of $V_i$ as  $v_0 v_1 \dots v_{\frac n 2 -1}$. 
   We define these permutations by setting for $i\in\{0,1,2,3\}$ and $k \in \{0,\dots, \frac n 4 -1\}$
   \begin{equation}\label{eq.sigmadef}
       (v_{i,2k}, u_{i,2k+1}, v_{i,2k+1}, u_{i,2k}) = C_{i,k}.
   \end{equation}
   Here for $k \ne 0$
   \begin{align*}
     & C_{i,k} = 
       \left\{
\begin{array}{l l}     
    (4k-1, 4k, 2n-4k-1,2n-4k) &  i=0; \\
    (4k, 4k+1, 2n-4k,2n-4k+1) &  i=1; \\
    (4k+1, 4k+2,2n-4k+1,2n-4k+2) &  i=2; \\
    (4k+2, 4k+3, 2n-4k-2,2n-4k-1) &  i=3; \\
\end{array}\right.
   \end{align*}
   and for $k=0$
\begin{align*}
    &C_{i,k} =
       \left\{
   \begin{array}{l l}     
       (i-1, i, n+i-1, n+i) &  i \in \{0,1,2\}; \\
    (2, 3, 2n-2, 2n-1) &  i=3. \\
    \end{array}\right.
 \end{align*}
 Let $\mathcal{F}_i = \{C_{i,k}: k \in \{0, \dots, \frac n 4 - 1\}\}$. Note that $\mathcal{F}_i$ is the family $\mathcal{F}(n/2)$ of special faces of $\mathcal{R}_i$ defined in (\ref{eq.cylinder}). 
 Importantly, each special face of $\mathcal{F}_i$ is made of exactly two arcs in $P_i$. The family $\mathcal{F}_i$ can alternatively
 be seen as a matching of the elements of $P_i$.

   Let us now describe how  $\mathcal{R}_0, \dots, \mathcal{R}_3$ are combined into an embedding of $K_{n,n}$.
   By the definition of $G_i$, any vertex $v \in \mathbb{Z}_{2n}$ belongs to exactly two of the four graphs, $G_{v \bmod 4}$ and $G_{(v+1) \bmod 4}$.
   Let $y_v = (y_{v,0}y_{v,1}\dots y_{v,\frac n 2 -1})$ and $z_v = (z_{v,0}z_{v,1}\dots z_{v,\frac n 2 - 1})$ be
   the rotations at $v$ in their respective embeddings. Without loss of generality we can assume that $y_{v,\frac n 2 - 1}=v-1$ and $z_{v,0} = v+1$. Note that 
the sets of elements of $y_v$ and $z_v$ are disjoint and partition the odd (even) vertices
of $\mathbb{Z}_{2n}$ if $v$ is even (odd). 

\begin{figure} 
    \begin{subfigure}[b]{.5\textwidth}
           \centering
           \includegraphics[scale=0.45]{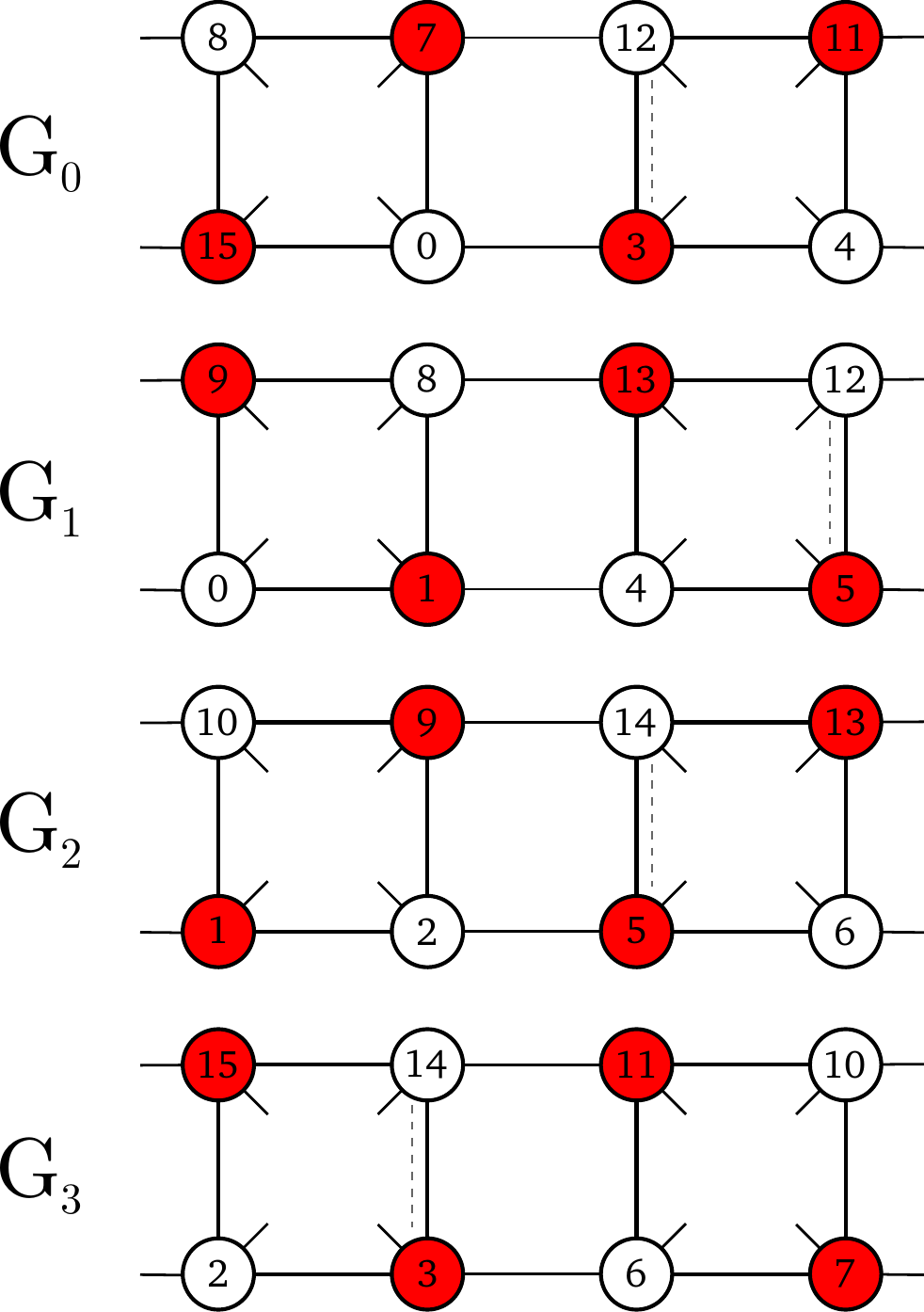}
           \caption{}
           \label{fig.circle8a}
        \end{subfigure}
        \begin{subfigure}[b]{.5\textwidth}
           \centering
           \raisebox{0.4 cm}{\includegraphics[scale=0.45]{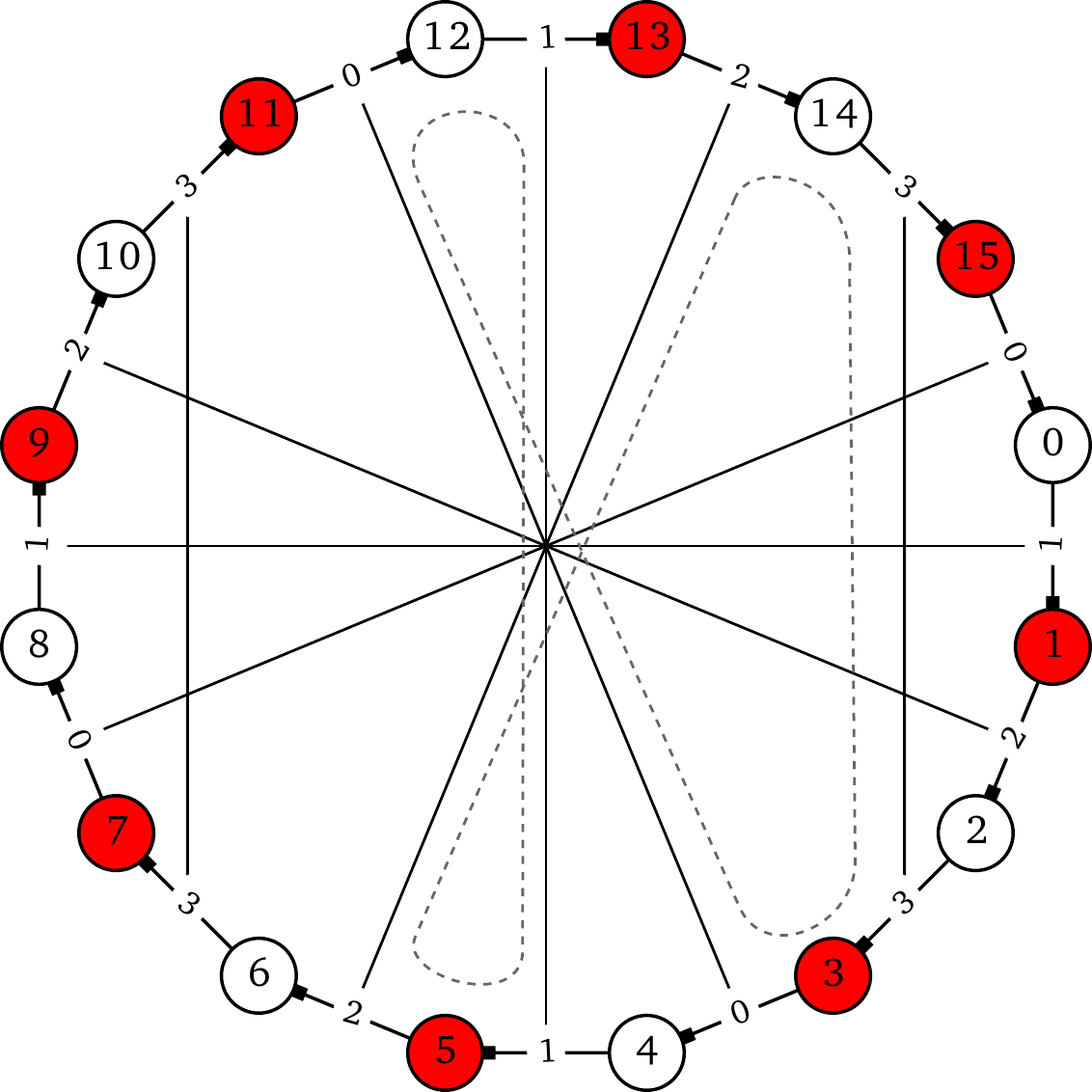}}
           \caption{}
           \label{fig.circle8b}
        \end{subfigure}
        \caption{The case $n=8$. (a) In $\mathcal{R}_i$, 
            the arcs of $P_i$ are paired to make up the 
            $\frac n 4$ disjoint special faces. The marks in their corners show
            where the rotations are cut in the concatenation step. 
            (b) The matchings of $P_i$, $i \in \{0,1,2,3\}$. The arc label $i$ indicates the subgraph $G_i$. The dashed cycle shows one of the new faces, $F_{1,1} = (3,14,5,12)$.}
    \label{fig.circle8}
\end{figure}

The rotation system $\mathcal{R} = \{\pi_v: v \in \mathbb{Z}_{2n}\}$ where
\begin{equation}\label{eq.rX}
 \pi_v = (y_{v,0}, \dots, y_{v, \frac n 2-1}, z_{v,0}, \dots, z_{v,\frac n 2-1}) = (\dots, v-1,v+1, \dots) 
\end{equation}
defines an embedding 
of $K_{n,n}$.
Notice the special role of vertices $y_{v,0}$, $y_{v, \frac n 2 - 1}=v-1$, $z_{v,0}=v+1$ and $z_{v,\frac n 2 -1}$: the embedding $\mathcal{R}_{v \bmod 4}$ contains a special face $(v-1, v, y_{v,0}, y_{v,0}+1)$.
Similarly $\mathcal{R}_{ (v+1) \bmod 4}$ contains a special face  $(v, v+1, z_{v,\frac n 2-1}-1, z_{v,\frac n 2-1})$.

What are the faces of $\mathcal{R}$? 
Note that if $\mathcal{R}_i$ has a rotation $a: (\dots, b, c, \dots)$ where $(b,a)$ and $(a,c)$ are not both
directed edges of a special face, then $\mathcal{R}$ also has a rotation $a: (\dots, b,c, \dots)$. Thus,
if $(a,b,c,d)$ is a non-special face of $\mathcal{R}_i$ then it is also a face of $\mathcal{R}$.
It follows by (\ref{eq.rX}) that the new faces of $\mathcal{R}$
are formed only from 
the directed edges of the cycles in $\cup_i \mathcal{F}_i$, see Figure~\ref{fig.circle8a}.

The right side of (\ref{eq.rX}) implies that $H$ is one of the new faces. 
Also, by (\ref{eq.sigmadef}) and (\ref{eq.rX}) we get that for each $k \in \{1, \dots, \frac n 4 - 1\}$ the embedding $\mathcal{R}$ contains a face 
\begin{align*}
    F_{1,k}  &= (4k-1, 2n-4k+2, 4k+1, 2n-4k)   
    \intertext{and a face}
    F_{2,k} &= (4k, 2n - 4k - 1, 4k+2, 2n - 4k + 1).
\end{align*}
Finally, the arcs of $C_{i,0}$ for $i \in \{0,1,2\}$ not lying on $H$, together with the unused edges 
$(2n-1, 2)$ and $(n-1, n+2)$ of $C_{3, 0}$ and $C_{3, \frac n 4 - 1}$ respectively 
yield 
a cycle $(0, n-1, n+2, 1, n, 2n-1, 2, n+1)$, which we denote $C_8$. 

Write $\mathcal{F} = \{F_{1,k}:  k \in \{1, \dots, \frac n 4 -1\} \} \cup  \{F_{2,k}:  k \in \{1, \dots, \frac n 4 -1\} \}$
and note that the arcs of the cycles in $\{H, C_8\} \cup \mathcal{F}$ cover all edges of $\cup_i \mathcal{F}_i$.
Thus, the lengths of face boundaries of $\mathcal{R}$ are $2n, 8, 4, 4, \dots, 4$,
the number of faces is $2 + \frac {2 n^2 - 2n - 8} 4 = \frac {n (n-1)} 2$
and by Euler's formula, the genus is $\frac {n^2 - 3n + 4} 4 = L(n)$.

Interestingly, all the faces in $\{H, C_8\} \cup \mathcal{F}$ 
can be obtained by a face-tracing algorithm on a graph formed 
by placing $H$ on a circle and
connecting midpoints of the the elements of $P_i$ that are matched (i.e. lying on the same special face in $\mathcal{F}_i$) by a chord, see Figure~\ref{fig.circle8b}.

\emph{The case $n \equiv 2 \;(\bmod\; 4)$}.
For $i \in \{0,1,2,3\}$ set
    \begin{align*}
        \tilde{P}_i = &\{ (t-1, t): t \equiv i \, (\bmod\; 4),  t \in \{0, \dots, n-1\} \}  
       \\ &\quad\quad\bigcup \{ (t-1, t): (t-n) \equiv i (\bmod\; 4),  t \in \{n, \dots, 2n-1\} \}.
   \end{align*}
   Let $P_i = \tilde{P}_i$ for $i\in\{1,2\}$ but $P_0 = \tilde{P}_0 \setminus \{(2n-1,0), (n-1,n)\}$ and  $P_3 = \tilde{P}_3 \cup \{(2n-1, 0), (n-1, n)\}$. Define $U_i = \{u: uv \in P_i\}$ and $V_i = \{v: uv \in P_i\}$ for $i\in\{0,1,2\}$, but  $U_3 = \{u: uv \in \tilde{P}_3\} \cup \{0,n\}$ and $V_3 = \{v: uv \in \tilde{P}_3\} \cup \{n-1, 2n-1\}$. Note that $|U_i|=|V_i| = n_i$ is even: $n_i = \frac {n-2} 2$ for $i\in\{0,2\}$ and $n_i = \frac {n+2} 2$ for $i\in\{1,3\}$.

   We again use Lemma~\ref{lem.Knm} for each $i\in\{0,1,2,3\}$ to construct an embedding $\mathcal{R}_i$ of a complete bipartite graph $G_i$ with parts $(U_i, V_i)$, with the property that the arcs in $P_i$ lie on the special faces of $\mathcal{R}_i$. 
   Specifically, we define the permutations $\bar{u}_i = u_{i,0} u_{i,1}\dots u_{i,n_i-1}$ and $\bar{v}_i = v_{i,0} v_{i,1}\dots v_{i,n_i-1}$
   (and, implicitly, the matchings of $P_i$) by setting the following cycles as special faces of the embedding $\mathcal{R}_i$, $i \in \{0,1,2,3\}$:
   \begin{align*} 
       &(v_{i,2k}, u_{i,2k+1}, v_{i,2k+1}, u_{i, 2k}) = 
       \left\{
\begin{array}{l l}     
    C_{i, k}, &  k \in S_i; \\
    (n, 2n-1, 0, n-1), & i=3, \, k = \frac {n-2} 4.
\end{array}\right.
       \end{align*}
       Here $u_{i, n_i} = u_{i,0}$, $v_{i,n_i}=v_{i,0}$ and $C_{i,k}$, $S_i$ are as follows (see also Figure~\ref{fig.circle10})

   \begin{center}
   \begin{tabular} {l | c | l}
       $i$ & $C_{i,k}$ & $S_i$ \\ \hline
       0 & $(4k + 3, 4k+4, 2n - 4k - 3, 2n-4k - 2)$ & $0, \dots, \frac {n-2} 4 - 1$ \\
       1 & $(4k, 4k+1, 2n - 4k - 2, 2n-4k-1)$ & $0, \dots, \frac {n-2} 4$ \\
       2 & $(4k+1, 4k+2, 2n - 4k - 5, 2n-4k-4)$ & $0, \dots, \frac {n-2} 4 - 1$ \\
       3 & $(4k+2, 4k+3, 2n - 4k - 4, 2n - 4k-3)$ & $0, \dots, \frac {n-2} 4 - 1$
   \end{tabular}
   \end{center}
   
   We will now combine $\mathcal{R}_0, \dots, \mathcal{R}_3$ similarly as in the case $n \equiv 0 \,(\bmod \, 4)$, but with one important difference:
   before concatenating rotations, we glue $\mathcal{R}_1$ and $\mathcal{R}_3$ 
   along the face on 4 shared vertices. 
   
   The gluing operation
   is carried out as follows.
   Let $\mathcal{F}_i$, $\mathcal{F}'_i$ be the sets of faces $\mathcal{F}(n_i)$, $\mathcal{F}'(n_i)$ defined in (\ref{eq.cylinder}) and (\ref{eq.cylinderprime}) respectively
   for $\mathcal{R}_i$. 
   Note that $\mathcal{F}'_1$ contains a face 
   \begin{align*}
       &F' = (u_{1, \frac n 2}, v_{1,0}, u_{1,0}, v_{1,\frac n 2}) = (n-1, 0, 2n-1, n),
       \intertext{and $\mathcal{F}_3$ contains a face}
       &F = (v_{3, \frac {n-2} 2}, u_{3, \frac n 2}, v_{3, \frac n 2}, u_{3, \frac {n-2} 2}) = (n, 2n-1, 0, n-1).
   \end{align*}
     Let $S=\{0,n-1,n,2n-1\}$. We have $V(G_1) \cap V(G_3) = S$.

     For $i \in \{0,1,2,3\}$ and $v \in V(G_i)$ let $p^i_v = (p^{i}_{v,0}, \dots, p^i_{v, n_i-1})$ be the rotation at $v$ in $\mathcal{R}_i$. 
   Choose the indices so that
   $p^i_{v,0}$ and $p^i_{v,n_i-1}$ are the neighbours of $v$ on the special face $F_i$ containing $v$. Also for $v \in S$ define
   permutations  $q^1_v = (q^1_{v,0}, \dots, q^1_{v,n_1-1})$, such 
   that $q^1_v$ is $p^1_v$ with $q^1_{v,0}$ and $q^1_{v, n_1-1}$ the neighbours of $v$ on the face $F'$.

\begin{figure}
    \begin{subfigure}[b]{.5\textwidth}
           \centering
           \includegraphics[scale=0.45]{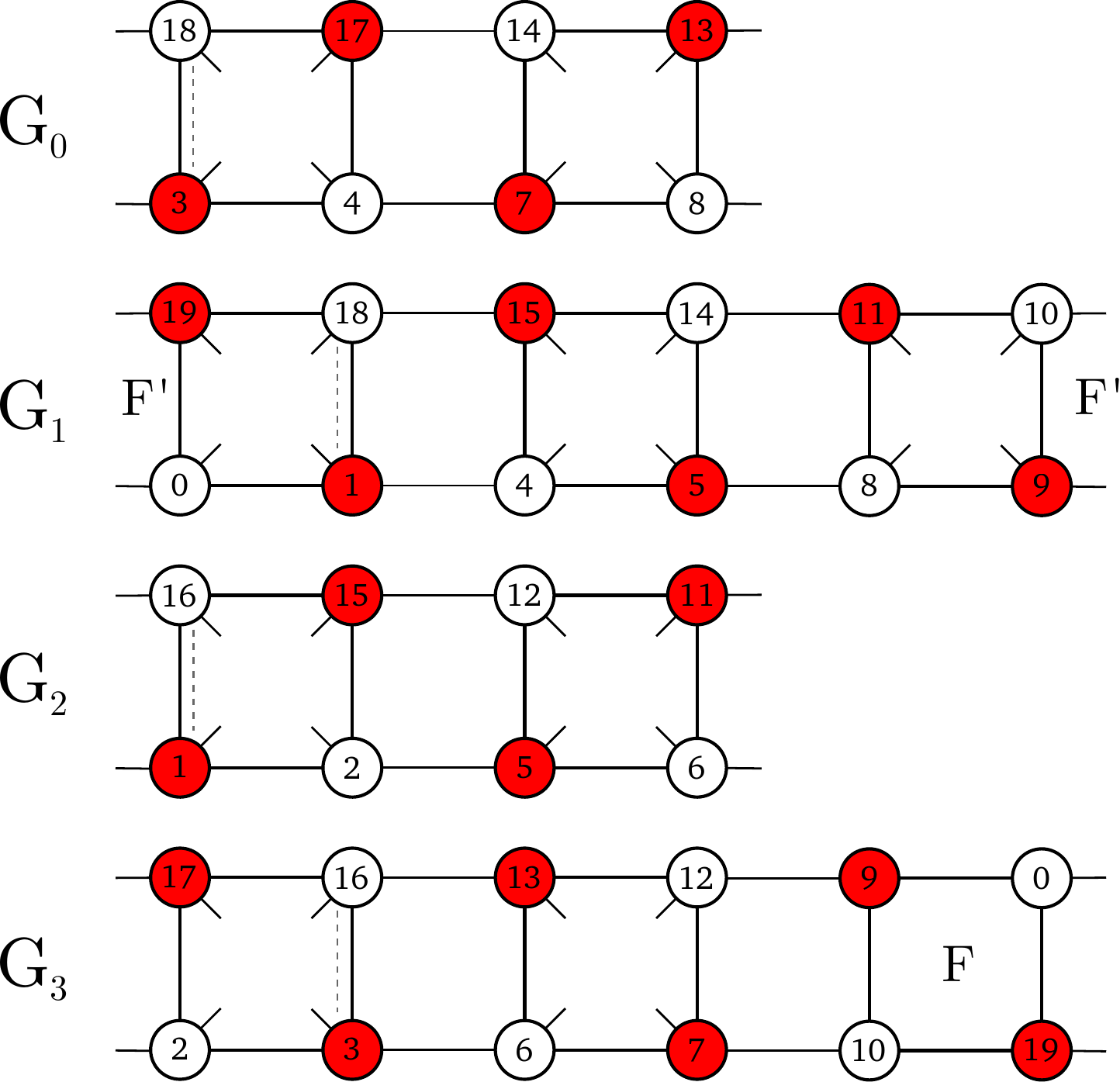}
           \caption{}
           \label{fig.circle10a}
        \end{subfigure}
        \begin{subfigure}[b]{.5\textwidth}
           \centering
           \raisebox{0.4 cm}{\includegraphics[scale=0.45]{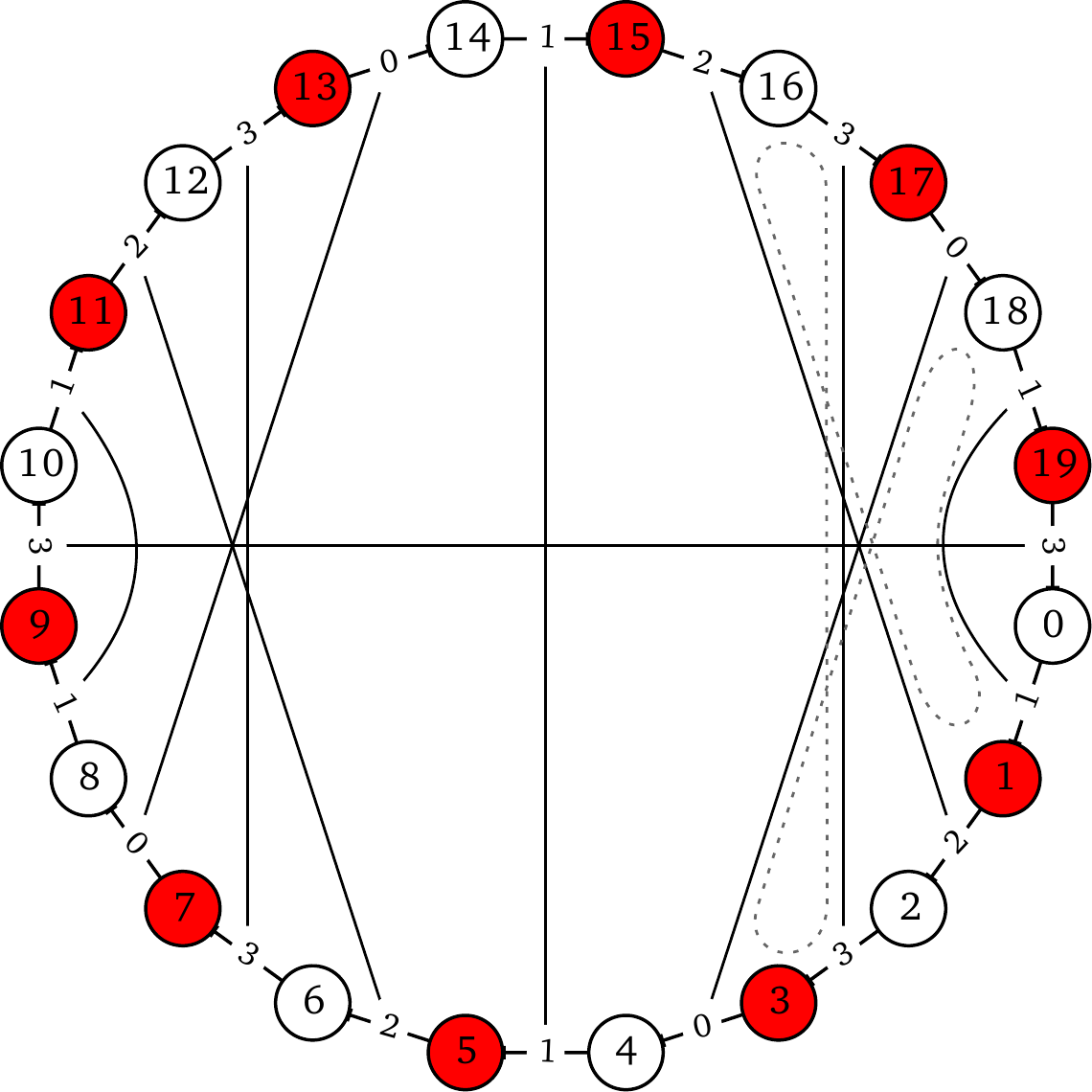}}
           \caption{}
           \label{fig.circle10b}
        \end{subfigure}
        \caption{The case $n=10$. (a) The embeddings of $G_i$ (only faces in $\mathcal{F}_i \cup \mathcal{F}_i'$ shown). (b) The matchings of arcs in $P_i$, $i=0,1,2,3$. The dashed edges show the face $F_{1,0} = (1,18,3,16)$.}
    \label{fig.circle10}
\end{figure}

   Let $G_{13}$ be a graph with vertex set $V(G_1) \cup V(G_3)$ and edge set $E(G_1) \cup E(G_3)$. This graph is bipartite, with parts $U_{13}$ and  $V_{13}$, where
   \[ 
       U_{13} = \{0,2,\dots,2n-2\}, \quad V_{13} = \{1,3, \dots, 2n-1\}.
   \]
   Since $F$ and $F'$ have opposite directions, $q^1_{v,0}=p^3_{v,n_3-1}$ and $ q^1_{v,n_1-1}=p^3_{v,0}$ for $v \in S$.
   Let the embedding of $G_{13}$ be $\mathcal{R}_{13} = \{p^{13}_v: v \in V(G_{13})\}$ where
   \begin{align*}
       &p^{13}_v = \left\{
\begin{array}{l l}  
    (q^1_{v,0}, q^1_{v,1}, \dots, q^1_{v, n_1-1}, p^3_{v,1}, p^3_{v,2}, \dots, p^3_{v,n_3-2}), &v \in S;
       \\ p^1_v, & v \in V(G_1) \setminus S;
       \\ p^3_v, & v \in V(G_3) \setminus S.
   \end{array}\right.
   \end{align*}
   The faces of $\mathcal{R}_{13}$ are the union of faces of $\mathcal{R}_1$ and $\mathcal{R}_3$ except $\{F, F'\}$; in particular
   they are all quadrangular. Call the faces in the set $\mathcal{F}_{13} = \mathcal{F}_1  \cup \left(\mathcal{F}_3  \setminus \{F\} \right)$, the \textit{special faces} of $\mathcal{R}_{13}$. $\mathcal{F}_{13}$ consists of $\frac {n-2} 2 + 1$ vertex-disjoint faces and covers $V(G_{13}) = \mathbb{Z}_{2n}$.
   Note that for $v \in S$, $p^{13}_v$ is a permutation of $V_{13}$ ($U_{13}$) 
   if $v$ is even (odd).

   Now the sets $V(G_0)$ and $V(G_2)$ partition $\mathbb{Z}_{2n} \setminus S$ into two subsets, each containing $\frac{n-2} 2$ even and $\frac{n-2} 2$ odd vertices. For $v \in \mathbb{Z}_{2n} \setminus S$ we let $t(v) \in \{0,2\}$ denote the index of the graph $G_{t(v)}$ it belongs to. It is easy to see that 
   the neighbours of $v$ in $G_{13}$ and the neighbours of $v$ in $G_{t(v)}$
partition the set of odd (even) vertices in $\mathbb{Z}_{2n}$ if $v$ is even (odd).
   
   Thus the rotation system $\mathcal{R} = \{\pi_v: v\in \mathbb{Z}_{2n}\}$ defined by
   \begin{align*}
     &\pi_v = \left\{
\begin{array}{l l}  
    p^{13}_v, &v \in S; \\
    (p^0_{v,0}, p^0_{v,1}, \dots, p^0_{v, \frac{n-2} 2}, p^{13}_{v,0}, p^{13}_{v,1}, \dots, p^{13}_{v, \frac{n+2} 2}), & v \in G_0; \\
    (p^2_{v,0}, p^2_{v,1}, \dots, p^2_{v, \frac{n-2} 2}, p^{13}_{v,0}, p^{13}_{v,1}, \dots, p^{13}_{v, \frac{n+2} 2}), & v \in G_2.
\end{array}\right.
   \end{align*}
   is an embedding of $K_{n,n}$ with parts $\{0,2, \dots, 2n-2\}$ and  $\{1, 3, \dots, 2n-1\}$.

   Suppose $v \in \mathbb{Z}_{2n} \setminus S$.
   By  the definition of $\mathcal{R}_i$ and $p^i_v$ we have that $p^i_{v, 0} = v+1$ if $v$ is even and $i \in \{1,3\}$ or $v$ is odd and $i \in \{0,2\}$.
   Similarly, $p^i_{v, n_i-1} = v-1$ if $v$ is odd and $i \in \{1,3\}$ or $v$ is even and $i \in \{0,2\}$. Thus
   $\pi_v = (\dots, v-1, v+1, \dots)$ if $v$ is even and $\pi_v = (v+1, \dots, v-1)$ if $v$ is odd.
   Now 
   $\mathcal{R}_1$ and hence also $\mathcal{R}_{13}$ 
   contains the face $C_{1, 0} = (0,1, 2n-2,2n-1)$
   and the face $C_{1, \frac {n-2} 4} = (n-2, n-1, n, n+1)$. This implies $p^{13}_v = (\dots, v-1, v+1, \dots)$ for $v \in S$.
   Thus  $\pi_v = (\dots, v-1, v+1, \dots)$  for all $v \in \mathbb{Z}_{2n}$ and $H$ is a face of $\mathcal{R}$.

   To complete the proof, we show that all faces of $\mathcal{R}$, apart from $H$, are of length $4$. Then by Euler's formula the genus $g$ of $\mathcal{R}$ satisfies $2 - 2g = 2n - n^2 + (2n^2 - 2n) / 4 + 1$, or $g = (n^2 - 3n + 2)/4 = L(n)$, as stated.

   As before, we only need to check the lengths of new faces in $\mathcal{R}$, that is,
   the faces formed from the arcs of the special faces of $\mathcal{R}_0$, $\mathcal{R}_2$ and $\mathcal{R}_{13}$. 
   Since these embeddings have $\frac {n-2} 4$, $\frac {n-2} 4$ and $\frac {n-2} 2 + 1$ special faces respectively,
   there are in total $4 (n - 1)$ such arcs. 
   
   Now $\mathcal{R}$ contains for each $k \in \{0, \dots, \frac {n-2} 4-1\}$ a pair of faces $F_{1,k}$ and $F_{2,k}$, where
   \begin{align*}
       & F_{1,k} = (4k+1, 2n-4k-2, 4k+3, 2n-4k-4),
    \\ & F_{2,k} = (4k+2, 2n-4k-5, 4k+4, 2n-4k-3).
   \end{align*}
   The faces in $\mathcal{F} = \{F_{1,k}: \{0, \dots, \frac {n-2} 4-1\}\} \cup \{F_{2,k}: \{0, \dots, \frac {n-2} 4-1\}\}$ are pairwise arc-disjoint and there are $2(n-2)$ of them. Thus $\{H\} \cup \mathcal{F}$ covers
   all $2(n-2) + 2n  = 4 (n-1)$ arcs from the special faces. This completes the proof.

   The faces in $\{H\} \cup \mathcal{F}$ can again be traced in a simple arc matching graph, see Figure~\ref{fig.circle10}. The illustration indicates one extra face $F=(0,n-1,n,2n-1)$ which is lost from $\mathcal{R}$ when gluing.
\end{proofof}

\vskip 1 cm

\textbf{Acknowledgement. } \textit{I would like to sincerely thank Rimvydas Krasauskas who introduced me to the road interchange problem and suggested to model it via graph embeddings.}

\end{document}